\documentclass[12pt,a4paper]{article}
\usepackage{enumerate}
\usepackage{amsmath}
\usepackage{amsfonts}
\usepackage{amssymb}
\usepackage[english]{babel}
\usepackage{amsthm}

\numberwithin{equation}{section}
\everymath{\displaystyle}
\title{Multitime hybrid differential games with \\ multiple integral functional}
\author{Constantin Udri\c ste, Elena-Laura Otob\^{i}cu, Ionel \c Tevy}
\begin{document}
\date{}
\renewcommand{\abstractname}{}

\maketitle

\begin{abstract}
A multiple integral functional is equivalent to a curvilinear integral functional,
 if the domain is a hyper-parallelepiped, but equivalence is only theoretical. The introduction
 of this kind of functionals in multitime optimal control problems, particularly in multitime differential games,
 is due to recent works of Udriste research group. The purpose of this paper
 is to introduce those ingredients that are necessary to formulate and to prove theorems
 about multitime differential games based on a multiple integral functional and an $m$-flow as constraint.
 The most important idea is to use a generating vector field for basic functions.
 The original results include: fundamental properties of multitime upper and lower values,
 viscosity solutions of multitime (dHJIU) PDEs,
representation formula of viscosity
solutions for a multitime (dHJ) PDE, and max-min representations.
\end{abstract}



\maketitle \numberwithin{equation}{section}
\newtheorem{theorem}{Theorem}[section]
\newtheorem{lemma}[theorem]{Lemma}
\newtheorem{corollary}[theorem]{Corollary}
\newtheorem{definition}[theorem]{Definition}
\newtheorem{remark}[theorem]{Remark}

\section{Multitime hybrid differential game with \\multiple integral functional}

Let $t=(t^\alpha)=(t^1,\ldots ,t^m)\in \Omega_{0T}\subset \mathbb{R}^m_+$, $\alpha =1,...,m,$ be an evolution multi-parameter (multi-time), $dt=dt^1\wedge \ldots \wedge dt^m$ is the volume element ($m$-form) in $\mathbb{R}^m_+,$ $\Omega_{0T}$ is the $m$-dimensional parallelepiped fixed by the diagonal opposite points $0=(0^1,\ldots,0^m)$ and $T=(T^1,\ldots,T^m)$ which is equivalent to the closed interval $0\leq t\leq T$ via the product order on $\mathbb{R}^m_+,$ a $C^2$ state vector
$x:\Omega_{0T}\rightarrow \mathbb{R}^n, x(t)=(x^i(t)),$ a $C^1$
control vector $u:\Omega_{0T}\rightarrow U\subset \mathbb{R}^{p}$, $u(t)=(u^a(t))$,
for the first equip of $p$ players (who wants to maximize), a $C^1$ control vector
$v:\Omega_{0T}\rightarrow V\subset \mathbb{R}^{q},$ $v(t)=(v^b(t))$,
for the second equip of $q$ players (who wants to minimize),
$u(\cdot)=\Phi(\cdot,\eta_1(\cdot)), v(\cdot)=\Psi(\cdot,\eta_2(\cdot)),$
a running cost $L(t,x(t),u(t),v(t))$ as a
nonautonomous continuous Lagrangian, a terminal cost (penalty term) $g(x(T))$
and the $C^1$ vector fields $X_\alpha=(X_\alpha^i)$ satisfying the complete
integrability conditions (CIC) $D_\beta X_\alpha=D_\alpha X_\beta)$ (m-flow type problem).

In this paper, a multitime  hybrid differential game is given by a multitime dynamics
(PDE system contolled by two controllers) and a target including a multiple integral functional.
Our new approach here is to define and use the generating vector field of the value of the differential game.
This original idea is coming from the multitime optimal theory developed in [9]-[19].

We want to analyse a multitime differential game whose Bolza payoff is
the sum between a multiple integral (volume) and a function of the final
event (the terminal cost) and whose evolution PDE is an m-flow:\\

{\it find
$$\min_{v(\cdot)\in V}\max_{u(\cdot)\in U}I(u(\cdot),v(\cdot))=\int_{\Omega_{0T}} L (s,x(s),u(s),v(s))ds + g(x(T)),$$
subject to the Cauchy problem
$$\frac{\partial x^i}{\partial s^\alpha}(s)=X^i_\alpha(s,x(s),u(s),v(s)),$$
$$x(0)=x_0,\, s\in \Omega_{0T}\subset \mathbb{R}_+^m,\, x\in \mathbb{R}^n.$$}

Let $u=(u^a)$, $a=1,...,p$, $v=(v^b)$, $b=1,...,q$. Let $D_\alpha$
be the total derivative operator and $[X_{\alpha},X_{\beta}]$ be the bracket of vector fields.
Suppose the piecewise complete integrability conditions (CIC)
$$ \left( \frac{\partial X_{\alpha}}{\partial u^a}\delta^{\gamma}_{\beta} -  \frac{\partial X_{\beta}}{\partial u^a}\delta^{\gamma}_{\alpha}\right)\frac{\partial u^a}{\partial s^{\gamma}}+\left( \frac{\partial X_{\alpha}}{\partial v^b}\delta^{\gamma}_{\beta} -  \frac{\partial X_{\beta}}{\partial v^b}\delta^{\gamma}_{\alpha}\right)\frac{\partial v^b}{\partial s^{\gamma}}=\left[ X_{\alpha},X_{\beta}\right] + \frac{\partial X_{\beta}}{\partial s^{\alpha}} - \frac{\partial X_{\alpha}}{\partial s^{\beta}}, $$
are satisfied throughout.

We vary the starting multitime and the initial point. We obtain a larger family of
similar multitime problems based on the functional
$$I_{t,x}(u(\cdot),v(\cdot))=\int_{\Omega_{tT}} L (s,x(s),u(s),v(s))ds+g(x(T))$$
and the multitime evolution constraint (Cauchy problem for first order PDEs system)
$$\frac{\partial x^i}{\partial s^\alpha}(s)=X^i_\alpha(s,x(s),u(s),v(s)),$$
$$x(t)=x,\, s\in \Omega_{tT}\subset \mathbb{R}_+^m,\, x\in \mathbb{R}^n.$$

We assume that, for some constant 1-form $A=(A_\alpha)$ and all $t\in \Omega_{0T}, x, \hat{x}\in \mathbb{R}^n, u\in U,v\in V,$ each vector field $X_\alpha:\Omega_{0T}\times \mathbb{R}^n\times U\times V\rightarrow \mathbb{R}^n$ is uniformly continuous and satisfies
$$\left\{\begin{array}{ll}
\Vert X_\alpha(t,x,u,v\Vert\leqslant A_\alpha\\
\Vert X_\alpha(t,x,u,v)-X_\alpha(t,\hat{x},u,v)\Vert\leqslant A_\alpha\Vert x-\hat{x}\Vert,
\end{array}\right.$$
Suppose the functions
$$g:\mathbb{R}^n\rightarrow \mathbb{R}, \quad L:\Omega_{0T}\times \mathbb{R}^n\times U\times V\rightarrow \mathbb{R}$$
are uniformly continuous and satisfy the boundedness conditions
$$\left\{\begin{array}{ll}
\vert g(x)\vert\leqslant B\\
\vert g(x)-g(\hat{x})\vert\leqslant B\Vert x-\hat{x}\Vert,
\end{array}\right.$$
$$\left\{\begin{array}{ll}
\vert L(t,x,u,v\vert\leqslant C\\
\vert L(t,x,u,v)-L(t,\hat{x},u,v)\vert\leqslant C\Vert x-\hat{x}\Vert,
\end{array}\right.$$
for constants $B, C$ and all $t\in \Omega_{0T},\,\, x, \hat{x}\in \mathbb{R}^n,\,\, u\in U,v\in V.$

\section{Control sets and value functions}

Here we include: control sets, strategies, value functions, and generating vector fields.

\begin{definition}
(i) The set
$$\mathcal {U}(t)=\left\lbrace u:\mathbb{R}^m_+\rightarrow U \vert \ u(\cdot) \mathrm{ \ is \ measurable \ and \ satisfies \ CIC}\right\rbrace $$
is called \textbf{the control set for the first equip of players}.

(ii) The set
$$\mathcal {V}(t)=\left\lbrace v:\mathbb{R}^m_+\rightarrow V \vert \ v(\cdot) \mathrm{ \ is \ measurable \ and \ satisfies \ CIC}\right\rbrace $$
is called \textbf{the control set for the second equip of players}.
\end{definition}

\begin{definition}
(i) A map $\Phi:\mathcal {V}(t)\rightarrow \mathcal {U}(t)$ is called \textbf{a strategy for the first equip of players}, if the equality $v(\tau)=v^\star(\tau),\, t\leq \tau \leq s \leq T$ implies $\Phi[v](\tau)=\Phi[v^\star](\tau).$

 (ii) A map $\Psi:\mathcal {U}(t)\rightarrow \mathcal {V}(t)$ is called \textbf{a strategy for the second equip of players}, if the equality $u(\tau)=u^\star(\tau),\,t\leq \tau \leq s \leq T$ implies $\Psi[u](\tau)=\Psi[u^\star](\tau).$
\end{definition}

Let $ \mathcal{A}(t) \mathrm {\ be \textbf{\ the set of strategies for the first equipe of players}}$
 and $\mathcal{B}(t) \mathrm {\ be \textbf{\ the set of strategies for the second equip of players}}.$

\begin{definition}
(i) The function
$$ m(t,x)=\min_{\Psi\in \mathcal{B}} \max_{u(\cdot)\in \mathcal{U}} I_{t,x}[u(\cdot),\Psi[u](\cdot)]$$
is called \textbf{the multitime lower value function}.

(ii) The function
$$M(t,x)=\max_{\Phi\in \mathcal{A}} \min_{v(\cdot)\in \mathcal{V}} I_{t,x}[\Phi[v](\cdot),[v](\cdot)]$$
is called \textbf{the multitime upper value function}.
\end{definition}

The most important ingredient in our theory is the idea of {\bf generating vector field} (see \cite{[19]}).

\begin{definition}
Let $D_{\alpha}$ be the total derivative.
A vector field $\textbf{u}=(\textbf{u}^{\alpha}(t,x))$ is called \textbf{a generating vector field} of the function $u(t,x),$ if
$$u(T,x(T))=C_{hyp}+u(t,x(t))+\int_{\Omega_{tT}}D_{\alpha}\textbf{u}^{\alpha}(s,x(s))\,ds.$$
\end{definition}

\begin{remark}
Two multitime Lagrangians which differs by a total divergence term have the same Euler-Lagrange PDEs.
\end{remark}

\begin{definition} Let $D_{\alpha}$ be the total derivative.

(i) The vector field $\textbf{m}(t,x)=(\textbf{m}^{\alpha}(t,x))$ is called
\textbf{the generating lower vector field} of the lower value function $m(t,x),$ if
$$m(T,x(T))=c_{hyp}+m(t,x(t))+\int_{\Omega_{tT}}D_{\alpha}\textbf{m}^{\alpha}(s,x(s))ds.$$

(ii) The vector field $\textbf{M}(t,x)=(\textbf{M}^{\alpha}(t,x))$ is called \textbf{the generating upper vector field} of the upper value function $M(t,x),$ if
$$M(T,x(T))=C_{hyp}+M(t,x(t))+\int_{\Omega_{tT}}D_{\alpha}\textbf{M}^{\alpha}(s,x(s))ds.$$
\end{definition}

\section{Multitime dynamic programming \\optimality conditions}

Let us give explicit formulas for lower and upper value functions which represent in fact
multitime dynamic programming optimality conditions.

\begin{theorem}\textbf{(multitime dynamic programming optimality conditions)}
For each pair of strategies $(\Phi,\Psi),$ the lower and upper value functions can be written respectively in the form

\begin{equation}\begin{split}
m(t,x)\ & =\min_{\Psi\in \mathcal{B}(t)} \max_{u\in \mathcal{U}(t)}\bigg\{  \int_{\Omega_{tt+h}} L (s,x(s),u(s),\Psi[u](s))ds\bigg\} \\& +m(t+h,x(t+h))
 \end{split}\end{equation}

and

\begin{equation}\begin{split}
M(t,x)\ &  =\max_{\Phi\in \mathcal{A}(t)}\min_{v\in \mathcal{V}(t)}\bigg\{  \int_{\Omega_{tt+h}} L (s,x(s),\Phi [v](s),v(s))ds\bigg\} \\ & +M(t+h,x(t+h)),
\end{split}\end{equation}
for all $(t,x) \in \Omega_{tT}\times \mathbb{R}^n$ and all $h\in \Omega_{0T-t}.$
\end{theorem}

\begin{proof} To confirm the previous statement for the lower value function, we introduce a new function

\begin{equation}\begin{split}
w(t,x)\ &=\min_{\Psi\in \mathcal{B}(t)} \max_{u\in \mathcal{U}(t)}\bigg\{  \int_{\Omega_{tt+h}} L (s,x(s),u(s),\Psi[u](s))ds\bigg\} \\ & +m(t+h,x(t+h)).
\end{split}\end{equation}

It will be enough to prove that the lower value function $m(t,x)$ satisfies two inequalities, $m(t,x)\leq w(t,x)+2\varepsilon$ and $m(t,x)\geq w(t,x)-3\varepsilon,\, \forall \varepsilon >0$.

\begin{enumerate}[i)]

\item Let us prove the first inequality. For $\varepsilon >0,$ there exists a strategy $\Upsilon\in \mathcal{B}(t)$ such that

\begin{equation}\begin{split}\label{eq.7}
w(t,x)\ &\geqslant \max_{u\in \mathcal{U}(t)}\bigg\{  \int_{\Omega_{tt+h}} L (s,x(s),u(s),\Upsilon[u](s))ds \\ & +m(t+h,x(t+h))\bigg\}-\varepsilon.
\end{split}\end{equation}

We shall use the state $x(\cdot)$ which solves the (PDE), with initial condition $\overline{x}=x(t+h),$ on $\Omega_{tT}\setminus \Omega_{tt+h},$ for each $\overline{x}\in \mathbb{R}^n.$ The following equality
\begin{equation}\begin{split}
m(t+h,\overline{x}) \ & =\min_{\Psi\in \mathcal{B}(t+h)} \max_{u\in \mathcal{U}(t+h)}\bigg\{ \int_{\Omega_{t+hT}} L (s,x(s),u(s),\Psi[u](s))ds  \\ & +g(x(T))\bigg\}
\end{split}\end{equation}
holds. Thus there exists a strategy $\Upsilon_{\overline{x}}\in \mathcal{B}(t+h)$ such that

\begin{equation}\begin{split}\label{eq.8}
 m(t+h,\overline{x})\ & \geqslant \max_{u\in \mathcal{U}(t+h)}\bigg\{  \int_{\Omega_{t+hT}} L (s,x(s),u(s),\Upsilon_{\overline{x}}[u](s))ds \\& + g(x(T))\bigg\}-\varepsilon.
\end{split}\end{equation}

Define the strategy
$$\Psi\in\mathcal{B}(t), \Psi[u](s)\equiv
\left\{\begin{array}{ll}
\Upsilon[u](s) & s\in \Omega_{tt+h}\\
\Upsilon_{\overline{x}}[u](s) & s\in \Omega_{tT}\setminus\Omega_{tt+h},
\end{array}\right.$$

for each control $u\in \mathcal{U}(t).$ Replacing the inequality $\eqref{eq.8}$ in the inequation $\eqref{eq.7}$ for any $u\in \mathcal{U}(t)$, we can write

$$w(t,x)\geqslant   \int_{\Omega_{tT}} L (s,x(s),u(s),\Psi[u](s))ds+g(x(T))-2\varepsilon.$$

Going from side to side and applying maximum, we obtain

$$ \max_{u\in \mathcal{U}(t)}\left\lbrace  \int_{\Omega_{tT}} L (s,x(s),u(s),\Psi[u](s))ds+g(x(T))\right\rbrace\leq w(t,x)+2\varepsilon.$$
By the definition of the lower value function, we have
$$m(t,x)\leq w(t,x)+2\varepsilon.$$

\item For the reverse inequality, there exists a strategy $\Psi\in\mathcal{B}(t)$ for which we can write the inequality
\begin{equation}\label{eq.9}
m(t,x)\geqslant \max_{u\in \mathcal{U}(t)}\left\lbrace  \int_{\Omega_{tT}} L (s,x(s),u(s),\Psi[u](s))ds+g(x(T))\right\rbrace-\varepsilon.
\end{equation}

The definition of $w(t,x)$ implies
\begin{equation}\begin{split}
w(t,x) \ & \leqslant \max_{u\in \mathcal{U}(t)}\bigg\{  \int_{\Omega_{tt+h}} L (s,x(s),u(s),\Psi[u](s))ds \\& +m(t+h,x(t+h))\bigg\}
\end{split}\end{equation}
and consequently there exists a control $u^1\in \mathcal{U}(t)$ such that
\begin{equation}\begin{split}\label{eq.10}
w(t,x) \ & \leqslant   \int_{\Omega_{tt+h}} L (s,x(s),u^1(s),\Psi[u^1](s))ds  \\&  +m(t+h,x(t+h))+\varepsilon.
\end{split}\end{equation}

Define a new control $${u^\star}\in \mathcal{U}(t), {u^\star}(s)\equiv
\left\{\begin{array}{ll}
u^1(s) & s\in \Omega_{tt+h}\\
 u(s) & s\in \Omega_{tT}\setminus\Omega_{tt+h},
\end{array}\right.$$
for a control $u\in \mathcal{U}(t+h)$ and then define the strategy ${\Psi}^\star\in\mathcal{B}(t+h), \Psi^\star[u](s)\equiv\Psi[{u^\star}](s), s\in \Omega_{tT}\setminus\Omega_{tt+h}.$
We find the inequality
\begin{equation}\begin{split}
 \ & m(t+h,x(t+h)) \\ & \leq\max_{u\in \mathcal{U}(t+h)}\left\lbrace \int_{\Omega_{tt+h}} L(s,x(s),u(s),\Psi^\star[u](s))ds+g(x(T))\right\rbrace
\end{split}\end{equation}
and so there exists the control $u^2\in \mathcal{U}(t+h)$ for which
\begin{equation}\begin{split}\label{eq.11}
 \ & m(t+h, x(t+h))\\&  \leq \int_{\Omega_{tT}\setminus \Omega_{tt+h}} L(s,x(s),u^2(s),\Psi^\star[u^2](s))ds+g(x(T)) +\varepsilon.
\end{split}\end{equation}

Define a new control
$$u\in \mathcal{U}(t), u(s)\equiv
\left\{\begin{array}{ll}
u^1(s) & s\in \Omega_{tt+h}\\
u^2(s) & s\in \Omega_{tT}\setminus\Omega_{tt+h}.
\end{array}\right.$$

Then the inequalities $\eqref{eq.10}$ and $\eqref{eq.11}$ yield
$$ w(t,x)\leq \int_{\Omega_{tT}} L(s,x(s),u(s),\Psi[u](s))ds+g(x(T)) +2\varepsilon,$$
and so $\eqref{eq.9}$ implies the inequality
$$ w(t,x)\leq m(t,x)+3\varepsilon.$$

Since $\varepsilon>0$ is arbitrary, this inequality and  $m(t,x)\leq w(x,t)+2\varepsilon$ complete the proof.
\end{enumerate}

\end{proof}

\section{Boundedness and continuity of \\values functions}
Now we add boundedness and continuity properties of lower and upper values functions.
A basic idea is to replace the Cauchy problem with associated curvilinear integral equation.

\begin{theorem}\textbf{(boundedness and continuity of values functions)} The lower value function $m(t,x)$ and the upper value function $M(t,x)$ satisfy the boundedness conditions
$$\vert m(t,x)\vert, \vert M(t,x)\vert\leq D$$
$$\vert m(t,x)-m(\hat{t},\hat{x})\vert, \vert M(t,x)-M(\hat{t},\hat{x}\vert\leq E\,\, vol (\Omega_{\hat{t}\,t})+ D\,\Vert x-\hat{x}\Vert ,$$
for some constant $D,E$ and for all $t, \hat{t}\in \Omega_{0T}, x, \hat{x} \in \mathbb{R}^n.$
\end{theorem}

\begin{proof} Because the two value functions have analogous definitions,
we prove only the statement for upper value function $M(t,x).$

Since $\vert g(x)\vert\leqslant  B, \vert L(t,x,u,v)\vert\leqslant C$, we find

\begin{equation}\begin{split}
\vert I_{t,x}(u(\cdot),v(\cdot))\vert \ & =\bigg\vert \int_{\Omega_{tT}} L (s,x(s),u(s),v(s))ds+g(x(T))\bigg\vert \\& \leq \bigg\vert \int_{\Omega_{tT}} L (s,x(s),u(s),v(s))ds \bigg\vert + \vert g(x(T))\vert \\& \leq \int_{\Omega_{tT}} \vert L (s,x(s),u(s),v(s))\vert  ds  +\vert g(x(T))\vert \\& \leq C \int_{\Omega_{tT}}  ds +  B \leq C \,\,vol(\Omega_{0T}) + B=D \\ & \Longrightarrow \vert M(t,x)\vert\leq D,
\end{split}\end{equation}
for all $u(\cdot)\in \mathcal{U}(t),v(\cdot)\in \mathcal{V}(t).$

Let $x_1,x_2\in \mathbb{R}^n, t_1, t_2 \in \Omega_{0T}.$ For $\varepsilon>0$ and the strategy $\Phi\in \mathcal{A}(t_1),$ we have

\begin{equation}\label{eq:1}
M(t_1,x_1)\leq \min_{v \in \mathcal{V}(t_1)} I(\Phi[v],v)+\varepsilon.
\end{equation}

Define the control
$$\overline{v}\in \mathcal{V}(t_1),\overline{v}(s)\equiv
\left\{\begin{array}{ll}
{v}^1(s) & s\in \Omega_{0t_2}\setminus\Omega_{0t_1}\\
{v}(s) & s\in \Omega_{0T}\setminus\Omega_{0t_2},
\end{array}\right.$$
for any $v \in \mathcal{V}(t_2)$ and some $v^1 \in V$ and for each $v\in \mathcal{V}(t_2), \underline{\Phi}\in \mathcal{A}(t_2)$ (the restriction of $\Phi$ over $\Omega_{0T}\setminus \Omega_{0t_1})$ by $\underline{\Phi}[v]=\Phi[\overline{v}], s\in \Omega_{0T}\setminus\Omega_{0t_2}.$

Take the control $v\in \mathcal{V}(t_2)$ such that

\begin{equation}\label{eq:2}
M(t_2,x_2)\geq  I(\underline{\Phi}[v],v)-\varepsilon.
\end{equation}

By the inequality $\eqref{eq:1},$ we deduce

\begin{equation}\label{eq:3}
M(t_1,x_1)\leq  I(\Phi[\overline{v}],\overline{v})+\varepsilon.
\end{equation}

We know that the (unique, Lipschitz) solution $x(\cdot)$ of the Cauchy problem
$$\left\{\begin{array}{ll}
\frac{\partial x^i}{\partial s^\alpha}(s)=X^i_\alpha(s,x(s),u(s),v(s))\\
x(t)=x, \, s\in \Omega_{tT}\subset \mathbb{R}_+^m,\, x\in \mathbb{R}^n,\, i=\overline{1,n},\, \alpha =\overline{1,m},
\end{array}\right.$$
is the response to the controls $u(\cdot), v(\cdot)$ for $s\in \Omega_{0T}.$

We choose $x_1(\cdot)$ as solution of the Cauchy problem
$$\left\{\begin{array}{ll}
\frac{\partial x^i_1}{\partial s^\alpha}(s)=X^i_\alpha(s,x_1(s),\Phi[\overline{v}],\overline{v})\\
x_1(t_1)=x_1, \quad s\in \Omega_{0T}\setminus \Omega_{0t_1}
\end{array}\right.$$
and $x_2(\cdot)$ as solution of the Cauchy problem
$$\left\{\begin{array}{ll}
\frac{\partial x^i_2}{\partial s^\alpha}(s)=X^i_\alpha(s,x_2(s),\underline{\Phi}[v],v(s))\\
x_2(t_2)=x_2, \quad s\in \Omega_{0T}\setminus \Omega_{0t_2}.
\end{array}\right.$$
Equivalently, $x_1(\cdot)$ is solution of curvilinear integral equation
$$x_1(s)= x_1(t_1) + \int_{\Gamma_{t_1s}}X_\alpha(\sigma,x_1(\sigma),\Phi[\overline{v}](\sigma),\overline{v}(\sigma))\,d\sigma^\alpha$$
and $x_2(\cdot)$ is solution of curvilinear integral equation
$$x_2(s)= x_2(t_2) + \int_{\Gamma_{t_2s}}X_\alpha(\sigma,x_2(\sigma),\Phi[\overline{v}](\sigma),\overline{v}(\sigma))\,d\sigma^\alpha.$$
It follows that

$$\Vert x_1(t_2)-x_1 \Vert = \Vert x_1(t_2)-x_1(t_1)\Vert \leq \Vert A\Vert \,\ell(\Gamma_{t_1t_2}).$$

Because $v=\overline{v}$ and $\underline{\Phi}[v]=\Phi[\overline{v}],$ for $s\in \Omega_{0T}\setminus\Omega_{0t_2},$ we find the estimation

\begin{equation}\begin{split}
\Vert x_1(s)-x_2(s)\Vert \ & \leq \Vert x_1(t_1)-x_2(t_2)\Vert + \Vert \int_{\Gamma_{t_1t_2}}\cdots\Vert  \\ & \leq \Vert A \Vert \ell(\Gamma_{t_1t_2})+ \Vert x_1-x_2\Vert,\,\, \hbox{on}\,\, t_2\leq s\leq T.
\end{split}\end{equation}

Thus the inequalities $\eqref{eq:2}$ and $\eqref{eq:3}$ imply

$$
 M(t_1,x_1)-M(t_2,x_2) \leq  I(\Phi[\overline{v}],\overline{v})-I(\underline{\Phi}[v],v)+2\varepsilon
$$
$$\leq \Big\vert \int_{\Omega_{t_1t_2}} L(s,x_1(s),\Phi[\overline{v}](s),\overline{v}(s))\,ds$$
$$+\int_{\Omega_{t_2T}} (L (s,x_1(s),\underline{\Phi}[v](s),v(s))  -L (s,x_2(s),\underline{\Phi}[v](s),v(s)))\,ds$$
 $$+g(x_1(T))-g(x_2(T))+2\varepsilon\Big\vert
$$
$$\leq\int_{\Omega_{t_1t_2}} \vert L (s,x_1(s),\Phi[\overline{v}](s),\overline{v}(s))\vert\, ds$$
$$+\int_{\Omega_{t_2T}} \vert (L (s,x_1(s),\underline{\Phi}[v](s),v(s))  -L (s,x_2(s),\underline{\Phi}[v](s),v(s)))\vert \,ds$$
$$+\vert g(x_1(T))-g(x_2(T))\vert +2\varepsilon$$
$$\leq  C \, vol(\Omega_{t_1t_2}) + (C \, vol(\Omega_{0T})+B)\,\Vert x_1-x_2\Vert +2\varepsilon.$$

Since $\varepsilon$ is arbitrary, we obtain the inequality
\begin{equation}\label{eq:7}
M(t_1,x_1)-M(t_2,x_2)\leq E\,vol(\Omega_{t_1t_2}) +D \,\Vert x_1-x_2\Vert.
\end{equation}

Let $\varepsilon>0$ and choose the strategy $\Phi\in \mathcal{A}(t_2)$ such that

\begin{equation}\label{eq:4}
M(t_2,x_2)\leq \min_{v \in \mathcal{V}(t_2)} I(\Phi[v],v)+\varepsilon.
\end{equation}

For each control $v \in \mathcal{V}(t_1)$ and $s\in \Omega_{0T}\setminus\Omega_{0t_2},$ define the control $\underline{v}\in \mathcal{V}(t_2), \underline{v}(s)=v(s).$

For some $u^1 \in U,$ we define the strategy $\overline{\Phi}\in \mathcal{A}(t_1)$ (the restriction of $\Phi$ to $\Omega_{0T}\setminus\Omega_{0t_2}$) by

$$\overline{\Phi}[{v}]=
\left\{\begin{array}{ll}
u^1 & s\in \Omega_{0t_2}\setminus\Omega_{0t_1}\\
\Phi[\underline{v}] & s\in \Omega_{0T}\setminus\Omega_{0t_2}.
\end{array}\right.$$

Now choose a control $v\in \mathcal{V}(t_1)$ so that

\begin{equation}\label{eq:5}
M(t_1,x_1)\geq  I(\overline{\Phi}[v],v)-\varepsilon.
\end{equation}

By the inequality $\eqref{eq:4},$ we have

\begin{equation}\label{eq:6}
M(t_2,x_2)\leq  I(\Phi[\underline{v}],\underline{v})+\varepsilon.
\end{equation}

We choose $x_1(\cdot)$ as solution of the Cauchy problem
$$\left\{\begin{array}{ll}
\frac{\partial x_1^i}{\partial s^\alpha}(s)=X^i_\alpha(s,x_1(s),\overline{\Phi}[v],v(s)), s\in \Omega_{0T}\setminus\Omega_{0t_1} \\
x_1(t_1)=x_1, \quad s\in \Omega_{0T}\setminus \Omega_{0t_1}
\end{array}\right.$$
and $x_2(\cdot)$ as solution of the Cauchy problem
$$\left\{\begin{array}{ll}
\frac{\partial x_2^i}{\partial s^\alpha}(s)=X^i_\alpha(s,x_2(s),\Phi[\underline{v}],\underline{v}(s)), s\in \Omega_{0T}\setminus\Omega_{0t_2}\\
x_2(t_2)=x_2, \quad s\in \Omega_{0T}\setminus \Omega_{0t_2}.
\end{array}\right.$$

It follows that
$$\Vert x_1(t_2)-x_1 \Vert = \Vert x_1(t_2)-x_1(t_1)\Vert \leq \Vert A\Vert \,\ell(\Gamma_{t_1t_2}).$$

For $s\in \Omega_{0T}\setminus\Omega_{0t_2}, v=\underline{v}$ and $\overline{\Phi}[v]=\Phi[\underline{v}],$
we find the estimation
\begin{equation}\begin{split}
\Vert x_1(s)-x_2(s)\Vert \ & \leq \Vert x_1(t_1)-x_2(t_2)\Vert + \Vert \int_{\Gamma_{t_1t_2}}\cdots\Vert  \\ & \leq \Vert A \Vert \ell(\Gamma_{t_1t_2})+ \Vert x_1-x_2\Vert,\,\, \hbox{on}\,\, t_2\leq s\leq T.
\end{split}\end{equation}

Thus, the relations $\eqref{eq:5}$ and $\eqref{eq:6}$ imply

$$M(t_2,x_2)-M(t_1,x_1) \leq I(\Phi[\underline{v}],\underline{v}])-I(\overline{\Phi}[v],v])+2\varepsilon$$
  $$
  \leq   \bigg\vert - \int_{\Omega_{t_1t_2}} L (s,x_1(s),\overline{\Phi}[{v}](s),{v}(s))ds$$
  $$+\int_{\Omega_{t_2T}} (L(s,x_2(s),{\Phi}[\underline{v}](s),\underline{v}(s))-L (s,x_1(s),{\Phi}[\underline{v}](s),\underline{v}(s)))ds
 $$
 $$ +g(x_2(T))-g(x_1(T))+2\varepsilon \bigg\vert$$
 $$\leq  \int_{\Omega_{t_1t_2}} \vert L (s,x_1(s),\overline{\Phi}[{v}](s),{v}(s))\vert \,ds$$
  $$+\int_{\Omega_{t_2T}} \vert L(s,x_2(s),{\Phi}[\underline{v}](s),\underline{v}(s))-L (s,x_1(s),{\Phi}[\underline{v}](s),\underline{v}(s))\vert\, ds
 $$
 $$ +\vert g(x_2(T))-g(x_1(T))\vert+2\varepsilon $$
 $$\leq  C\, vol(\Omega_{t_1t_2}) + (C \, vol(\Omega_{0T})+B)\,\Vert x_2-x_1\Vert +2\varepsilon.$$

Since $\varepsilon$ is arbitrary, we obtain
\begin{equation}
M(t_2,x_2)-M(t_1,x_1)\leq E\,vol(\Omega_{t_1t_2}) +D \,\Vert x_1-x_2\Vert.
\end{equation}

By this inequality and $\eqref{eq:7},$ we prove the continuity of the lower and upper value functions.
\end{proof}

\section{Viscosity solutions of \\ multitime (dHJIU) PDEs}
The key original idea is that the generating
upper vector field $\textbf{M}=(\textbf{M}^\alpha)$
or the generating lower vector field $\textbf{m}=(\textbf{m}^\alpha)$
are solutions of (dHJIU) PDEs.

\begin{remark}
The generating upper vector field was introduced by the relation
\begin{equation}
\begin{split}
M(t+h) \ & =M(t)+C_{hyp}+\int_{\Omega_{tt+h}} D_\alpha \textit{\textbf{M}}^\alpha ds\\ & \Rightarrow M(t)-M(t+h)=-C_{hyp}-\int_{\Omega_{tt+h}} D_\alpha \textit{\textbf{M}}^\alpha\, ds.
\end{split}\end{equation}

The multitime dynamic programming optimality condition gives
\begin{equation}\begin{split}
\ & M(t) =\max_{\Phi\in \mathcal{A}(t)}\min_{v\in \mathcal{V}(t)}\bigg\{  \int_{\Omega_{tt+h}} L (s,x(s),\Phi [v](s),v(s))ds\bigg\} +M(t+h)\\ & \Rightarrow M(t)- M(t+h) =\max_{\Phi\in \mathcal{A}(t)}\min_{v\in \mathcal{V}(t)}\bigg\{  \int_{\Omega_{tt+h}} L (s,x(s),\Phi [v](s),v(s))\,ds\bigg\}. \end{split}\end{equation}

These two equalities suggest a multitime divergence Hamilton-Jacobi-Isaacs-Udri\c ste (dHJIU) PDE.
\end{remark}

\begin{theorem}\textbf{(PDEs for generating upper vector field, resp. lower vector field)}

The generating upper vector field $\textbf{M}=(\textbf{M}^\alpha(t,x))$ and the generating lower vector field $\textbf{m}=(\textbf{m}^\alpha(t,x))$ are the viscosity solutions of
\begin{itemize}
\item the multitime divergence type upper Hamilton-Jacobi-Isaacs-Udri\c ste \\(dHJIU) PDE
$$\frac{\partial \textbf{M}^{\alpha}}{\partial t^\alpha}(t,x)+\min_{v\in V} \max_{u\in U} \left\lbrace  \frac{\partial \textbf{M}^{\alpha}}{\partial x^i}(t,x) X_\alpha^i(t,x,u,v)+L(t,x,u,v)\right\rbrace =0,$$
which satisfies the terminal condition $\textbf{M}^\alpha(T,x)=g^\alpha(x),$
\item the multitime divergence type lower Hamilton-Jacobi-Isaacs-Udri\c ste \\(dHJIU) PDE
$$\frac{\partial \textbf{m}^\alpha}{\partial t^\alpha}+\max_{u \in U} \min_{v \in V} \left\lbrace  \frac{\partial \textbf{m}^\alpha}{\partial x^i}(t,x) X_\alpha^i(t,x,u,v)+L(t,x,u,v)\right\rbrace =0,$$
which satisfies the terminal condition $\textbf{m}^\alpha(T,x)=g^\alpha(x).$
\end{itemize}
\end{theorem}

\begin{remark} If we introduce the so-called upper and lower Hamiltonian defined respectively by
$$H^+(t,x,p)=\min_{v\in \mathcal{V}} \max_{u \in \mathcal{U}}\lbrace p_i^\alpha(t) X_\alpha^i(t,x,u,v)+L(t,x,u,v)\rbrace,$$
$$H^-(t,x,p)=\max_{u\in \mathcal{U}} \min_{v\in \mathcal{V}}\lbrace p_i^\alpha(t) X_\alpha^i(t,x,u,v)+L(t,x,u,v)\rbrace,$$
then the multitime (dHJIU) PDEs can be written in the form
$$\frac{\partial \textbf{M}^\alpha}{\partial t^\alpha}(t,x)+H^+\left( t,x,\frac{\partial \textbf{M}}{\partial x}(t,x)\right) =0$$
and
$$\frac{\partial \textbf{m}^\alpha}{\partial t^\alpha}(t,x)+H^-\left( t,x,\frac{\partial \textbf{m}}{\partial x}(t,x)\right) =0.$$
\end{remark}

The proof will be given in another paper.

\section{Representation formula of viscosity \\solution for multitime (dHJ) PDE}

In this section, we want to obtain a representation formula for
the viscosity solution $\textbf{M}=(\textbf{M}^\alpha(t,x))$ of the multitime (dHJ) PDE

\begin{equation}
 \frac{\partial \textbf{M}^\alpha}{\partial t^\alpha}+H\left( t,x,\frac{\partial \textbf{M}}{\partial x}(t,x)\right) =0, (t,x)\in \Omega_{0T}\times \mathbb{R}^n,\alpha=\overline{1,m},
\end{equation}

\begin{equation}
\textbf{M}^\alpha(0,x)=g^\alpha(x), x\in \mathbb{R}^n.
\end{equation}
On the other hand, the upper value function $M(t,x)$, generated by $\textbf{M}=(\textbf{M}^\alpha(t,x))$, satisfies the inequalities
\begin{equation}\label{eq:11}
\left\{\begin{array}{ll}
\vert M(t,x)\vert\leq D\\
\vert M(t,x)-M(\hat{t},\hat{x}\vert\leq E\,\, vol (\Omega_{\hat{t}\,t})+ D\,\Vert x-\hat{x}\Vert,
\end{array}\right. \end{equation}
for some constant $E,\,D$ (for $m=1,$ see also [4]).

Also, we assume that $g:\mathbb{R}^n \rightarrow \mathbb{R}^m, H:\Omega_{0T} \times \mathbb{R}^n\times \mathbb{R}^{nm}\rightarrow \mathbb{R},$ satisfy the inequalities
$$\left\{\begin{array}{ll}
\Vert g (x)\Vert\leq B\\
\Vert g(x)-g (\hat{x})\Vert\leq B \Vert x-\hat{x}\Vert
\end{array}\right.$$
and
\begin{equation}\label{eq:8}
\left\{\begin{array}{ll}
\vert H(t,x,0)\vert\leq K\\
\vert H (t,x,p)-H(\hat{t},\hat{x},\hat{p})\vert\leq K (vol(\Omega_{t\hat{t}})+\Vert x-\hat{x}\Vert +\Vert p-\hat{p}\Vert).
\end{array}\right. \end{equation}
By the norm of the matrix $p=(p^\alpha_i)$, we understand $||p||= \sqrt{\delta^{ij} \delta_{\alpha\beta} p^\alpha_i p^\beta_j}$.
Otherwise, in this paper, all norms of indexed variables are norms of vectors associated by re-indexing.

\begin{lemma}\label{l-2}
Let
\begin{equation}\label{eq:10}
\left\{\begin{array}{ll}
U=B(0,1)\subset \mathbb{R}^{n}\\
V={\cal B}(0,P)\subset \mathbb{R}^{mn}\\
X_\alpha(u)=Q_\alpha u,\, Q=(Q^i_{\alpha j}),\, ||Q||=K\\
L(t,x,u,v)=H (t,x,v)-<Q u,v>.
\end{array}\right.
\end{equation}
Suppose the Hamiltonian $H$ is a Lipschitz function. For some constant radius
$P>0$ and for each $t\in \Omega_{0T}, x \in \mathbb{R}^n,$ we have
$$H (t,x,p)=\max_{v \in {V}}\min_{u \in {U}}\left\lbrace p^\alpha_i(t) X^i_\alpha(u) + L(t,x,u,v)\right\rbrace ,$$
if $\Vert p\Vert \leq P.$
\end{lemma}

\begin{proof} By the assumption $H (t,x,v)-H(t,x,p)\leq K \Vert p-v\Vert,\,K = ||Q||$, by Cauchy-Schwarz formula
and by the condition $||u||\leq 1$, we have
\begin{equation}\begin{split}
H (t,x,p)\ & =\max_{v\in {V}} \left\lbrace H(t,x,v) - K\Vert p-v\Vert\right\rbrace \\ & =\max_{v \in {V}}\min_{u\in {U}}\left\lbrace H(t,x,v)+<Q u,p-v>\right\rbrace,
\end{split}\end{equation}
for any $x\in \mathbb{R}^n$.
\end{proof}

\textbf{Max-min representation of a Lipschitz function as positive homogeneous functions} (for m=1, see also [2], [3]).

\begin{lemma}
Let $H$ be a Lipschitz m-form which is homogeneous in the matrix $p$, i.e.,
$$H(t,x,\lambda p)=\lambda H(t,x,p),\,\lambda \geq 0.$$
Then there exist compact sets $U \subset \mathbb{R}^{2n}$, $V \subset \mathbb{R}^{2mn}$ and
vector fields $$X_\alpha:\Omega_{0T}\times \mathbb{R}^n\times U\times V\rightarrow \mathbb{R}^n,\,\, \alpha=1,...,m,$$
satisfying
$$\Vert X_\alpha(x)-X_\alpha(\hat{x})\Vert \leq A_\alpha \Vert x-\hat{x}\Vert,$$
for each $\alpha,$ and such that
$$H (t,x,p)=\max_{v \in V}\min_{u \in U}\left\lbrace p^\alpha_i(t)X^i_\alpha(t,x,u,v) \right\rbrace ,$$
for all $t\in \Omega_{0T},\,\, x\in \mathbb{R}^n, \,\,p \in \mathbb{R}^{mn}.$
\end{lemma}

\begin{proof} Let $u=(u^1,u^2)$ be a $2n$-dimensional control, $v=(v^1,v^2)$ be a $2mn$-dimensional control.
We introduce the notations
\begin{equation}\label{eq:9}
\left\{\begin{array}{ll}
U= B(0,1)\times B(0,1)\subset \mathbb{R}^{2n}\\
V= {\cal B}(0,1)\times {\cal B}(0,1)\subset \mathbb{R}^{2mn}\\
L(t,x,u^1,v^1)=H(t,x,v^1)-<Q u^1,v^1>\\
X_\alpha(t,x,u,v)=Q_\alpha u^1+C v^2_\alpha+(L(t,x,u^1,v^1)-C)v^2_\alpha.
\end{array}\right.
\end{equation}

According to Lemma $\eqref{l-2}$ and the assumptions $\eqref{eq:9},$ if $\Vert\eta\Vert =1,$ we have
\begin{equation}\begin{split}
H (t,x,\eta)\ &  =\max_{v^1 \in V^1}\min_{u^1\in U^1}\left\lbrace <Q u^1,\eta> + L (t,x,u^1, v^1)\right\rbrace.
\end{split}\end{equation}
for $U^1=B(0,1)\subset \mathbb{R}^n$, $V^1={\cal B}(0,1)\subset \mathbb{R}^{mn}$.

For any non-zero matrix $p=(p^\alpha_i),$ we can write
\begin{equation}\begin{split}
H(t,x,p)\ & =\Vert p\Vert H \left( t,x,\frac{p}{\Vert p\Vert}\right) \\ & =\max_{v^1\in V^1}\min_{u^1\in U^1}\left\lbrace <Q u^1,p>+L (t,x,u^1,v^1)\Vert p\Vert\right\rbrace .
\end{split}\end{equation}
Then, if we choose $C>0$ such that $\vert L\vert\leq C,$ we find

\begin{equation}\begin{split}
H (t,x,p)\ &  =\max_{v^1 \in V^1}\min_{u^1\in U^1}\bigg\{ <Q u^1,p> +C\Vert p\Vert +( L (t,x,u^1, v^1)-C)\Vert p\Vert\bigg\}\\ & =\max_{v^1 \in V^1}\min_{u^1\in U^1}\max_{v^2 \in V^2}\min_{u^2\in U^2}\bigg\{ <Q u^1,p> +<C v^2, p>\\ &  +( L (t,x,u^1, v^1)-C) <v^2,p>\bigg\} \\ & =\max_{v \in V}\min_{u\in U}\bigg\{ <X(t,x,u,v),p>\bigg\} .
\end{split}\end{equation}

Now, interchanging $\min_{u^1\in U^1}$ and  $\max_{v^2 \in V^1},$ the result in Lemma follows.
\end{proof}



\end{document}